\documentclass[10pt,reqno]{amsart}
\usepackage{geometry}
\usepackage[utf8]{inputenc}

\usepackage{enumerate}
\usepackage{amstext,amsmath,amsthm,amsfonts,amssymb,amscd}
\usepackage{latexsym,mathrsfs,dsfont,euscript}
\usepackage{color}

\usepackage{hyperref} 
\hypersetup{
  colorlinks,%
  citecolor=blue,%
    filecolor=red,%
    linkcolor=red,%
    urlcolor=blue
}

\newtheorem{theorem}{Theorem}[section]
\newtheorem{proposition}[theorem]{Proposition}
\newtheorem{lemma}[theorem]{Lemma}

\theoremstyle{definition}

\theoremstyle{remark}

\numberwithin{equation}{section}

\newcommand{\abs}[1]{\left\vert#1\right\vert}

\newcommand{\proin}[2]{\left<#1,#2\right>}

\allowdisplaybreaks
\begin{document}
\title[]{Haar wavelet characterization of dyadic Lipschitz regularity}
%


\author[]{Hugo Aimar}
\author[]{Carlos Exequiel Arias}
\author[]{Ivana G\'{o}mez}
%


\begin{abstract}
We obtain a necessary and sufficient condition on the Haar coefficients of a real function $f$ defined on $\mathbb{R}^+$ for the Lipschitz $\alpha$ regu\-larity of $f$ with respect to the ultrametric $\delta(x,y)=\inf \{\abs{I}: x, y\in I; I\in\mathcal{D}\}$, where $\mathcal{D}$ is the family of all dyadic intervals in $\mathbb{R}^+$ and $\alpha$ is positive. Precisely, $f\in \textrm{Lip}_\delta(\alpha)$ if and only if $\abs{\proin{f}{h^j_k}}\leq C 2^{-(\alpha + \tfrac{1}{2})j}$, for some constant $C$, every $j\in\mathbb{Z}$ and every $k=0,1,2,\ldots$ Here, as usual $h^j_k(x)= 2^{j/2}h(2^jx-k)$ and $h(x)=\mathcal{X}_{[0,1/2)}(x)-\mathcal{X}_{[1/2,1)}(x)$.
\end{abstract}

	\textit{To Pola Harboure and Roberto Macías}
	
	\smallskip	
	\begin{verse}
		Arde de abejas el aguaribay, arde.
		
		Ríen los ojos, los labios, hacia las islas azules
		
		a través de la cortina 
		
		de los racimos
		
		pálidos.
		
		\hspace*{5.6cm}	Juan L. Ortiz
	\end{verse}
	
	\medskip

\maketitle

\section{Introduction}\label{sec:intro}

In \cite{HolTcha91} and \cite{HolTcha90}, see also \cite{Daubechiesbook}, M.~Holschneider and  Ph.~Tchamit\-chian provide characterizations of the Lipschitz $\alpha$ regularity of a function in $L^2(\mathbb{R})$ for $0<\alpha<1$ in terms of the behaviour of the continuous wavelet transform. The result is that a given function is Lipschitz $\alpha$ if and only if its continuous wavelet transform satisfies a power law in the absolute value of the scale parameter. Here Lipschitz $\alpha$ refers to the classical definition with respect to the usual metric in $\mathbb{R}$, i.e. $\abs{f(x)-f(y)}\leq C\abs{x-y}^\alpha$ for some constant $C>0$ and every $x$ and $y$ in $\mathbb{R}$. In \cite{AiBer96} these results are extended to more general moduli of regularity of functions when the basic wavelet is the Haar wavelet. The method used in \cite{AiBer96} provides the tool for the analysis of pointwise regularity through the discrete wavelet transform associated to dyadic scaling and integer translations of the Haar wavelet. The natural Lipschitz $\alpha$ class, in our setting, is defined through the dyadic distance instead of the usual one.

The result of this paper is contained in the next statement.

\begin{theorem}\label{thm:mainresult}
	Let $f$ be a real valued function in $L^1_{\emph{loc}}(\mathbb{R}^+)$. Let $h^j_k(x)=2^{j/2} h(2^jx-k)$ where $h(x)=\mathcal{X}_{[0,1/2)}(x)-\mathcal{X}_{[1/2,1)}(x)$, $j\in\mathbb{Z}$, $k=0, 1, 2, \ldots$, and $\proin{f}{h^j_k}=\int_{\mathbb{R}^+}f(x) h^j_k(x) dx$. Let $\alpha$ be any positive number. Then, the boundedness of the sequence
	\begin{equation*}
		\left\{2^{(\alpha + \tfrac{1}{2})j} \abs{\proin{f}{h^j_k}}: j\in\mathbb{Z}, k=0,1,2,\ldots\right\}
	\end{equation*}
	is equivalent to the essential boundedness of the quotients
	\begin{equation*}
		\frac{\abs{f(x)-f(y)}}{\delta^\alpha(x,y)}, \quad x\neq y,
	\end{equation*}
	where $\delta(x,y)=\inf\{\abs{I}: x, y\in I; I\in\mathcal{D}\}$ with $\mathcal{D}$ the family of all dyadic intervals in $\mathbb{R}^+$.
\end{theorem}

In Section~\ref{sec:dyadicdistanceHaarsystem} we introduce the basic facts and notation and Section~\ref{sec:proofmainresult} is devoted to the proof of Theorem~\ref{thm:mainresult}.

\section{Dyadic distance in $\mathbb{R}^+$ and the Haar system}\label{sec:dyadicdistanceHaarsystem}

The set of nonnegative real numbers is denoted here by $\mathbb{R}^+$. The family of all dyadic intervals in $\mathbb{R}^+$ is the disjoint union of the classes $\mathcal{D}^j$, $j\in\mathbb{Z}$, where $\mathcal{D}^j=\{I^j_k=[k2^{-j},(k+1)2^{-j}): k=0,1,2,\ldots\}$ are the dyadic intervals of level $j$. Notice that with this notation, when $j$ grows the partitions of $\mathbb{R}^+$ get refined and the intervals smaller. Since given two points $x$ and $y$ of $\mathbb{R}^+$ there exists some $j_0\in\mathbb{Z}$ such that $0\leq \max \{x,y\}<2^{-j_0}$, we have that $x, y\in I^{j_0}_0$. Hence, the class of all dyadic intervals $I\in\mathcal{D}$ such that $x$ and $y$, both, belong to $I$, is non-empty. So that, if $\abs{E}$ denotes the Lebesgue length of the measurable set $E$, we have that
\begin{equation*}
	\delta(x,y)=\inf \left\{\abs{I}: x, y\in I; I\in\mathcal{D}\right\}
\end{equation*}
is a well defined nonnegative real number. Even more, $\delta$ is an ultrametric in $\mathbb{R}^+$. In other words,
\begin{enumerate}[(i)]
	\item $\delta(x,y)=0$ if and only if $x=y$;
	\item $\delta(x,y)=\delta(y,x)$, for every $x$ and every $y$ both in $\mathbb{R}^+$;
	\item $\delta(x,z)\leq \max\{\delta(x,y),\delta(y,z)\}$ for every $x, y, z$ in $\mathbb{R}^+$.
\end{enumerate}
The triangle inequality follows from the properties of the family $\mathcal{D}$. In fact, given $x, y$ and $z$ in $\mathbb{R}^+$, let $I(x,y)$ and $I(y,z)$ denote the smallest dyadic intervals containing $x, y$ and $y, z$ respectively, then, one of these intervals contains the other because $y\in I(x,y)\cap I(y,z)\neq \emptyset$. Assume $I(x,y)\supseteq I(y,z)$, then $\delta(x,z)\leq \abs{I(x,y)}=\max \{\abs{I(y,z)},\abs{I(x,y)}\}=\max\{\delta(y,z),\delta(x,y)\}$.
In particular, $\delta$ is a metric in $\mathbb{R}^+$. Notice that $\abs{x-y}\leq \delta(x,y)$, but $\frac{\delta(x,y)}{\abs{x-y}}$ is unbounded. Hence every Lipschitz$(\alpha)$ function $f$ in the usual sense $(\abs{f(x)-f(y)}\leq C\abs{x-y}^\alpha)$ is also a $\textrm{Lip}_\delta(\alpha)$ function, i.e.
\begin{equation*}
	\abs{f(x)-f(y)}\leq C\delta^\alpha(x,y)
\end{equation*}
for some constant $C$ and every $x$ and $y$ in $\mathbb{R}^+$. On the other hand, there are $\textrm{Lip}_\delta(\alpha)$ functions which are not Lipschitz$(\alpha)$ in the classical sense. In fact, $\mathcal{X}_I$, $I\in\mathcal{D}$, is in the class $\textrm{Lip}_\delta(1)$. We also observe that in contrast with the class Lipschitz$(\alpha)$ for every $\alpha>1$, which is trivial, there exist non constant $\textrm{Lip}_\delta(\alpha)$ functions for every $\alpha>0$.

Let us now review the basic properties of the Haar system. Set $h^0_0(x)=\mathcal{X}_{[0,1/2)}(x) -\mathcal{X}_{[1/2,1)}(x)$ and $h^j_k(x)=2^{j/2}h^0_0(2^jx-k)$ for $j\in\mathbb{Z}$ and $k=0, 1, 2,\ldots$ The family $\mathscr{H}=\{h^j_k: j\in\mathbb{Z}, k=0,1,2,\ldots\}$ is the Haar system in $\mathbb{R}^+$. It is well known that $\mathscr{H}$ is an orthonormal basis for $L^2(\mathbb{R}^+)$. Since for each $I\in\mathcal{D}$ there is one and only one $h\in\mathscr{H}$ supported in $I$, we write sometimes $h_I$ to denote  the $h\in\mathscr{H}$ supported in $I\in\mathcal{D}$ and sometimes $I_h$ to denote the dyadic support of $h\in\mathscr{H}$. From the basic character of $\mathscr{H}$ in $L^2(\mathbb{R}^+)$ we have that, given $f\in L^2(\mathbb{R}^+)$,
\begin{equation*}
	f = \sum_{h\in\mathscr{H}} \proin{f}{h}h,
\end{equation*}
in the $L^2(\mathbb{R}^+)$-sense, with $\proin{f}{h}=\int_{\mathbb{R}^+}f(x) h(x) dx$. The sequence of coefficients $\{\proin{f}{h}: h\in\mathscr{H}\}$ is well defined even for functions in $L^1_{\textrm{loc}}(\mathbb{R}^+)$, since the Haar functions are bounded and have bounded support.

\section{Proof of Theorem~\ref{thm:mainresult}}\label{sec:proofmainresult}
The easy part of Theorem~\ref{thm:mainresult} follows as usual from the vanishing of the mean of the Haar functions. Let us state and prove it.

\begin{proposition}\label{propo:fLipdeltaimpliescoefficientsbounded}
	Let $f\in\textrm{Lip}_\delta(\alpha)$, $\alpha >0$. Set $[f]_{\textrm{Lip}_\delta(\alpha)}$ to denote the infimum of the constants $C>0$ such that $\abs{f(x)-f(y)}\leq C\delta^\alpha(x,y)$, $x, y\in\mathbb{R}^+$. Then $\abs{\proin{f}{h_I}}\leq [f]_{\textrm{Lip}_\delta(\alpha)} \abs{I}^{\alpha+\tfrac{1}{2}}$ for every $I\in\mathcal{D}$.
\end{proposition}
\begin{proof}
	For $I=[a_I,b_I)\in\mathcal{D}$ we have $\int_{\mathbb{R}^+}h_I(x) dx=0$, hence
	\begin{align*}
		\abs{\proin{f}{h_I}} &= \abs{\int_{\mathbb{R}^+} f(x)h_I(x) dx}\\
		&= \abs{\int_{\mathbb{R}^+}(f(x) -f(a_I))h_I(x) dx}\\
		&\leq \int_{I}\abs{f(x) -f(a_I)}\abs{h_I(x)} dx\\
		&\leq [f]_{\textrm{Lip}_\delta(\alpha)} \int_I \delta^{\alpha}(x,a_I)\abs{I}^{-\tfrac{1}{2}} dx\\
		&\leq [f]_{\textrm{Lip}_\delta(\alpha)}\abs{I}^{\alpha-\tfrac{1}{2}}\int_I dx\\
		&= [f]_{\textrm{Lip}_\delta(\alpha)}\abs{I}^{\alpha-\tfrac{1}{2}+1}\\
		&= [f]_{\textrm{Lip}_\delta(\alpha)}\abs{I}^{\alpha+\tfrac{1}{2}}.
	\end{align*}
\end{proof}

In order to prove that the size of the coefficients guarantee the regularity of $f$ we start by stating and proving a lemma. Given an interval $I\in\mathcal{D}$ we denote with $I^-$ and $I^+$ its left and right halves respectively. Notice that when $I\in\mathcal{D}^j$ then $I^-$ and $I^+$ both belong to $\mathcal{D}^{j+1}$.
Given a locally integrable  function $f$ we write $m_I(f)$ to denote the mean value of $f$ on $I\in\mathcal{D}$. In other words $m_I(f)=\tfrac{1}{\abs{I}}\int_I f(x) dx$.

\begin{lemma}\label{lemma:oscilationmI}
	Let $f\in L^1_{\emph{loc}}(\mathbb{R}^+)$. Then, for every $I\in\mathcal{D}$ we have 
	\begin{equation*}
		\abs{m_{I^-}(f)- m_{I^+}(f)} = 2\abs{I}^{-\tfrac{1}{2}} \abs{\proin{f}{h_I}}.
	\end{equation*}
\end{lemma}
\begin{proof}
	Let $I\in\mathcal{D}$ be given, then
	\begin{align*}
			\abs{m_{I^-}(f)- m_{I^+}(f)} &= \abs{\frac{2}{\abs{I}}\int_{I^-} f(x) dx-\frac{2}{\abs{I}}\int_{I^+} f(x) dx}\\
			&= 2\abs{I}^{-\tfrac{1}{2}}\abs{\int_I \abs{I}^{-\tfrac{1}{2}}\left(\mathcal{X}_{I^-}(x) -\mathcal{X}_{I^+}(x)\right) f(x)dx} \\
			&= 2\abs{I}^{-\tfrac{1}{2}} \left(\int_{\mathbb{R}^+} h_I(x) f(x) dx\right)\\
			&= 2\abs{I}^{-\tfrac{1}{2}} \abs{\proin{f}{h_I}}.
	\end{align*}
\end{proof}
\begin{proposition}\label{propo:coefficientsboundedimpliesfLipdelta}
	Let $f\in L^1_{\emph{loc}}(\mathbb{R}^+)$ be such that for some constant $A>0$ we have
	\begin{equation*}
		\abs{\proin{f}{h_I}}\leq A \abs{I}^{\alpha+\tfrac{1}{2}}
	\end{equation*}
for every $I\in\mathcal{D}$, then $f\in \textrm{Lip}_\delta(\alpha)$ and $[f]_{\textrm{Lip}_\delta(\alpha)}\leq C_\alpha A$ with $C_\alpha= \sup\{2,\frac{1}{2^\alpha -1}\}$.
\end{proposition}
\begin{proof}
	 Let $x<y$ be two points in $\mathbb{R}^+$. Let $I\in\mathcal{D}$ be the smallest dyadic interval containing $x$ and $y$. In other words $\abs{I}=\delta(x,y)$. Since $x<y$, necessarily $x\in I^-$ and $y\in I^+$. Set $I^x_1= I^-$ and $I^y_1=I^+$. Now let $I^x_2$ be the half of $I^x_1$ to which $x$ belongs, and $I^y_2$ the half of $I^y_1$ with $y\in I^y_2$. In general, once $I^x_l$ and $I^y_l$ are defined, we select $I^x_{l+1}$ as the only half of $I^x_l$ with $x\in I^x_{l+1}$ and $I^y_{l+1}$ as the only half of $I^y_l$ with $y\in I^y_{l+1}$. In this way for a fixed positive integer $k$ we have
	 \begin{equation*}
	 	I^x_k\subset I^{x}_{k-1}\subset \cdots\subset I^x_2\subset I^x_1\subset I,
	 \end{equation*} 
 and 
 \begin{equation*}
 	I^y_k\subset I^{y}_{k-1}\subset \cdots\subset I^y_2\subset I^y_1\subset I.
 \end{equation*} 
Hence
\begin{align*}
	f(x) - f(y) &= \left(f(x) - m_{I^x_k}(f)\right) +\\ 
	&\ + \left(m_{I^x_k}(f) - m_{I^x_{k-1}}(f)\right) + \cdots +
	\left(m_{I^x_2}(f) - m_{I^x_{1}}(f)\right) +\\
	&\ + \left(m_{I^x_1}(f) - m_{I^y_{1}}(f)\right) +\\
	&\ + \left(m_{I^y_1}(f) - m_{I^y_{2}}(f)\right) + \cdots 
	+ \left(m_{I^y_{k-1}}(f) - m_{I^y_{k}}(f)\right) \\
	&\ + \left(m_{I^y_k}(f) - f(y)\right).
\end{align*}
Then 
\begin{align*}
	\abs{f(x)-f(y)} &\leq \abs{f(x)-m_{I^x_k}(f)}+\\
	& + \sum_{l=2}^k\abs{m_{I^x_l}(f)-m_{I^x_{l-1}}(f)} +\\
	& + \abs{m_{I^x_1}(f)-m_{I^y_{1}}(f)}+\\
	& + \sum_{l=1}^{k-1}\abs{m_{I^y_l}(f)-m_{I^y_{l+1}}(f)} + \\
	& + \abs{m_{I^y_k}(f)-f(x))}\\
	&= I + II + III + IV + V.
\end{align*}
Let us start by bounding the central term $III$. Notice that $I^x_1=I^-$ and $I^y_1= I^+$, with $\abs{I}=\delta(x,y)$. Then by Lemma~\ref{lemma:oscilationmI},
\begin{align*}
	III &= \abs{m_{I^x_1}(f)-m_{I^y_{1}}(f)}\\
	&= \abs{m_{I^-}(f)-m_{I^+}(f)}\\
	&= 2 \abs{I}^{-\tfrac{1}{2}} \abs{\proin{f}{h_I}}\\
	&\leq 2 A \abs{I}^{-\tfrac{1}{2}} \abs{I}^{\alpha+\tfrac{1}{2}}\\
	&= 2A \abs{I}^\alpha\\
	&= 2A \delta^\alpha(x,y),
\end{align*}
which has the desired form. The terms $II$ and $IV$ can be handled in the same way, let us deal with $II$. Take a generic term of the sum $II$, and use again Lemma~\ref{lemma:oscilationmI}.
\begin{align*}
	\abs{m_{I^x_l}(f)-m_{I^x_{l-1}}(f)} &= \abs{\frac{1}{\abs{I^x_l}}\int_{I^x_l}f - \frac{1}{\abs{I^x_{l-1}}}\left(\int_{I^x_l}f + \int_{I^x_{l-1}\setminus I^x_l} f\right)}\\
	&= \abs{\frac{1}{2}\frac{1}{\abs{I^x_l}}\int_{I^x_l}f - \frac{1}{2} \frac{1}{\abs{I^x_{l-1}\setminus I^x_l}}\int_{I^x_{l-1}\setminus I^x_l} f}\\
	&= \frac{1}{2} \abs{m_{I^x_l}(f)-m_{I^x_{l-1}\setminus I^x_l}(f)}\\
	&=\frac{1}{2} 2 \abs{I^x_{l-1}}^{-\tfrac{1}{2}} \abs{\proin{f}{h_{I^x_{l-1}}}}\\
	&\leq A  \abs{I^x_{l-1}}^{-\tfrac{1}{2}}  \abs{I^x_{l-1}}^{\alpha+\tfrac{1}{2}}\\
	&= A \abs{I^x_{l-1}}^{\alpha}\\
	&= A \frac{2^\alpha}{2^{\alpha l}}\abs{I}^\alpha.
\end{align*}
Then 
\begin{align*}
	II &= \sum_{l=2}^k\abs{m_{I^x_l}(f)-m_{I^x_{l-1}}(f)}\\
	&\leq  A 2^\alpha \abs{I}^\alpha \sum_{l\geq 2}\frac{1}{2^{\alpha l}}\\
	&= \frac{A}{2^\alpha -1} \delta^\alpha(x,y).
\end{align*}
Also
\begin{equation*}
	IV\leq \frac{A}{2^\alpha -1} \delta^\alpha(x,y).
\end{equation*}
Let $C_\alpha=\sup\{2,\frac{1}{2^\alpha -1}\}$, then
\begin{equation*}
	\abs{f(x)-f(y)} \leq \abs{f(x)- m_{I^x_k}(f)} + AC_\alpha \delta^\alpha(x,y) + \abs{f(y)-m_{I^y_k}(f)}
\end{equation*}
uniformly in $k$. Now, from the differentiation theorem, we have that for almost all $x$ and almost all $y$,
\begin{equation*}
	m_{I^x_k}(f) \longrightarrow f(x); \quad k\to\infty
\end{equation*}
and
\begin{equation*}
	m_{I^y_k}(f) \longrightarrow f(y); \quad k\to\infty.
\end{equation*}
Hence, for those values of $x$ and $y$ in $\mathbb{R}^+$ we get the result
\begin{equation*}
	\abs{f(x)-f(y)} \leq A C_\alpha \delta^{\alpha}(x,y).
\end{equation*}
\end{proof}

Propositions~\ref{propo:fLipdeltaimpliescoefficientsbounded} and \ref{propo:coefficientsboundedimpliesfLipdelta} prove Theorem~\ref{thm:mainresult}.


\providecommand{\bysame}{\leavevmode\hbox to3em{\hrulefill}\thinspace}
\providecommand{\MR}{\relax\ifhmode\unskip\space\fi MR }
\providecommand{\MRhref}[2]{%
	\href{http://www.ams.org/mathscinet-getitem?mr=#1}{#2}
}
\providecommand{\href}[2]{#2}

\section*{Acknowledgements} 
This work was supported by the Ministerio de Ciencia, Tecnolog\'ia e Innovaci\'on-MINCYT in Argentina: Consejo Nacional de Investigaciones Cient\'ificas y T\'ecnicas-CONICET (grant PIP-2021-2023-11220200101940CO) and Universidad Nacional del Litoral (grant CAI+D 50620190100070LI).


\medskip

\noindent \textit{Affiliation:}
	\textsc{Instituto de Matem\'{a}tica Aplicada del Litoral ``Dra. Eleonor Harboure'', CONICET, UNL.}

\smallskip
\noindent \textit{Address:} \textmd{IMAL, CCT CONICET Santa Fe, Predio ``Dr. Alberto Cassano'', Colectora Ruta Nac.~168 km~0, Paraje El Pozo, S3007ABA Santa Fe, Argentina.}

\smallskip

\noindent \textit{E-mail address:} \verb|haimar@santafe-conicet.gov.ar|; \verb|carias@santafe-conicet.gov.ar|;\\ \verb|ivanagomez@santafe-conicet.gov.ar|

\end{document}